%% file: boundedness-paper.tex
\input{mixed-fields-preamble}

\title{A boundedness theorem for principal bundles on curves}
\author{Huai-Liang Chang, Shuai Guo, Jun Li, Wei-Ping Li, Yang Zhou}
\date{}

\begin{document}
\maketitle
\begin{abstract}
  Let $G$ be a reductive group acting on an affine scheme $V$. We study the set of principal
  $G$-bundles on a smooth projective curve $\mathscr C$ such that the associated
  $V$-bundle admits a section sending the generic point of $\mathscr C$ into the
  GIT stable locus $V^{\mathrm{s}}(\theta)$. We show that after fixing the degree
  of the line bundle induced by the character $\theta$, the set of such principal
  $G$-bundles is bounded.
  The statement of our theorem is made slightly more general so that we deduce from
  it the boundedness for $\epsilon$-stable quasimaps and $\Omega$-stable LG-quasimap.
\end{abstract}
\tableofcontents
% \addtocontents{toc}{\protect \setcounter{tocdepth}{1}}

\input{boundedness-intro-4}

\input{boundedness4}
% Boundendess4 is supposed to be the final version, with unless comments
% removed, and many common senses about algebraic groups also removed.

\bibliographystyle{alpha}
\bibliography{references.bib}

\end{document}

%% file: mixed-fields-preamble.tex
\documentclass[12pt, amsfonts, epsfig]{amsart}
\usepackage[foot]{amsaddr}
\usepackage{amsmath, amscd, amssymb}
\usepackage{graphicx, psfrag}
\usepackage[%
  width=16.5cm,
  left=3cm,
  right=3cm,
]{geometry}

\numberwithin{equation}{section}
\usepackage{bm}
\usepackage{stmaryrd}
\usepackage{graphpap, color}
\usepackage{mathrsfs}
\usepackage[shortlabels]{enumitem}
\usepackage{color}
\usepackage{cancel}
\usepackage{tikz-cd}
\usepackage{tikz}
\usepackage[mathscr]{eucal}
\usepackage{verbatim}
\usepackage[linktoc=all]{hyperref}
\hypersetup{
  colorlinks,
  citecolor=black,
  filecolor=black,
  linkcolor=black,
  urlcolor=black
}
\usepackage{latexsym}
\usepackage[all]{xy}

\def\beq{\begin{equation}}
\def\eeq{\end{equation}}

\xyoption{rotate}

\newtheorem{prop}{Proposition}[section]
\newtheorem{proposition}[prop]{Proposition}
\newtheorem{theo}[prop]{Theorem}
\newtheorem{lemm}[prop]{Lemma}

\newcommand{\git}{/\!\!/}
\newtheorem{definition}[prop]{Definition}

\newtheorem{defiprop}[prop]{Definition-Proposition}

\newtheorem{defi-prop}[prop]{Definition-Proposition}
\newtheorem{defi-theo}[prop]{Definition-Theorem}
\usepackage{lipsum}% http://ctan.org/pkg/lipsum
\renewcommand{\tilde}{\widetilde}

%% file: boundedness-intro-4.tex
\section{Introduction}
\label{sec:intro}

\subsection{Overview}
We work over the field $\mathbb C$ of complex numbers.

Let $\mathscr C$ be a connected smooth projective curve, $G$ be a reductive
group, $V$ be a finite-dimensional $G$-representation and $\theta \in \widehat{G}$ be a character of $G$ such
that the GIT stable  locus $V^{\mathrm{s}}(\theta)$ is non-empty.
A map from $\mathscr C$ into $[V/G]$ is by definition given by
a principal $G$-bundle $\mathscr P$ over $\mathscr C$,
together with a section of the associated
vector bundle
\[
  \mathscr P \times_{G} V = (\mathscr P \times V )/G \to \mathscr C.
\]
It is called a quasimap if the generic point of $\mathscr C$ is sent to the
stable locus $\mathscr P \times_{G} V^{\mathrm{s}}(\theta)$ \cite{ciocan2014stable}. We would like to
study the boundedness of the moduli of such quasimaps.
In the theory of quasimap it is not hard to bound the underlying curves
$\mathscr C$. So for simplicity let us fix $\mathscr C$.
Then it is easy to see that
once we also fix $\mathscr P$, the moduli of sections of $\mathscr P \times_{G} V$
is bounded. Thus, the key is to study the boundedness of the set of principal
bundles $\mathscr P$ which
admits such a section.

There are obvious topological invariants  to be fixed before achieving boundedness.
For any character $\chi \in \widehat{G}$, let $\mathbb C_{\chi}$ denote the $1$-dimensional
representation of $G$ on which the action is given by $\chi$.
Thus, given $\mathscr P$, $\mathscr P \times_{G} \mathbb C_{\theta}$ is a line
bundle on $\mathscr C$ and we denote its degree by $\deg_{\mathscr C}(\mathscr P \times_{G} \mathbb C_{\theta})$.
% Then $\mathscr P$ induces a
% group homomophism
% \[
%   \mathrm{deg}_{\mathscr P}: \widehat{G} \longrightarrow \mathbb Z,
%   \quad \chi \mapsto \deg_{\mathscr C}(\mathscr P \times_{G} \mathbb C_{\chi}),
% \]
% where $\deg_{\mathscr C}(\mathscr P \times_{G} \mathbb C_{\chi})$ denotes the
% degree of the line bundle $\mathscr P \times_{G} \mathbb C_{\chi}$ on
% $\mathscr C$.
The main result of the paper is the following.
\begin{theo}[Simplified version]
  \label{thm:simple}
  For any fixed integer $d$, 
  the set of principal $G$-bundles $\mathscr P$
  on $\mathscr C$ such that
  \begin{itemize}
  \item
    $\deg_{\mathscr C}(\mathscr P \times_{G} \mathbb C_{\theta}) = d$, and
  \item
    $\mathscr P \times_{G} V$ admits a section sending
    the generic point of $\mathscr C$ into $\mathscr P \times_{G} V^{\mathrm{s}}(\theta)$ 
  \end{itemize}
  is bounded.
\end{theo}
The main application of this paper is to complete the boundedness part in
\cite{mixed-fields-main} in the nonabelian case. Hence, we will
consider a slightly more complicated setup and prove the more general
Theorem~\ref{thm:main}. Nevertheless, the proofs are essentially the same.

In \cite{mixed-fields-main}, we introduced the $\Omega$-stability of
LG-quasimaps, generalizing the
construction of Mixed-Spin-P fields \cite{chang2019mixed}  and the
Gauged-Linear-Sigma-Model \cite{fan2017mathematical}.  We proved the
separatedness and properness of the moduli of
$\Omega$-stable LG-quasimaps in the abelian case.
The only missing part of \cite{mixed-fields-main} is the boundedness in the
nonabelian case, which is completed by applying Theorem~\ref{thm:main}.

Since stable quasimaps can be viewed as a special
case of stable LG-quasimaps \cite[Section~4.4]{mixed-fields-main},
as a corollary we also prove the boundedness
of stable quasimap.

In this paper, all
schemes are assumed to be locally of finite type over the complex numbers $\mathbb C$.
We refer the reader to \cite{mixed-fields-main} for the notion of LG-quasimaps and its
$\Omega$-stability, which is our major motivation. We will also recall what we
need in the following, for the sake of self-containedness of this paper.
In \cite{mixed-fields-main}, we have proved  that for a collection of stable LG-quasimaps of bounded
numerical data, the collection of their underlying curves is bounded. And by a
standard argument, to prove the boundedness for stable LG-quasimaps, it suffices to
prove that the family of underlying principal bundles is bounded. We have also
shown in \cite{mixed-fields-main} that the degrees of the line bundle $\mathscr P\times_{\Gamma}
\mathbb C_{\vartheta}$ on the irreducible components of $\mathscr C$ are bounded.
Hence the boundedness is reduced to the main theorem of this paper, which we now state.

Let $\Gamma$ be a reductive group acting on an affine scheme $V$. Let $\epsilon \in
\widehat{\Gamma}$ be a character of $\Gamma$ fitting into an exact sequence
\begin{equation}
  \label{eq:extension-of-groups}
  1 \to G \to \Gamma \overset{\epsilon}{\to} \mathbb G_m \to 1.
\end{equation}
Let $\vartheta \in \widehat\Gamma$ and $\theta = \vartheta|_{G}$.
Let $V^{\mathrm{s}}(\theta)$, $V^{\mathrm{ss}}(\theta)$, and $V^{\mathrm{un}}(\theta)$ be the
$\theta$-stable, semistable, and unstable loci, respectively.
We assume $V^{\mathrm{s}}(\theta)\ne \emptyset$, but do not require
$V^{\mathrm{s}}(\theta) = V^{\mathrm{ss}}(\theta)$ in this paper.

Let $\Sigma^{\mathscr C} \subset \mathscr C \to S$ be any family of genus-$g$
$n$-pointed twisted curves (c.f.~\cite[Section~4]{abramovich2002compactifying})
over a scheme $S$ of finite type over $\mathbb C$.
For any closed point $s\in S$, we denote by $\mathscr C_s$ the fiber over $s$.
\begin{theo}
  \label{thm:main}
  Let $d_1,d_2 \in \mathbb Z$, the set of principal $\Gamma$-bundles $\mathscr
  P_s$ on $\mathscr C_s$, for all $s\in S(\mathbb C)$, such that
  \begin{itemize}
  \item
    $\mathscr P_s
    \times_{\Gamma} V \to \mathscr C_s$ admits a section $\sigma$ sending
    the generic points of $\mathscr C_s$ into $\mathscr P_s \times_{\Gamma} V^{\mathrm{s}}(\theta)$, and
  \item
    for each irreducible component $\mathscr C^\prime \subset \mathscr C_{s}$, we have
    \[
      \deg(\mathscr P\times_{\Gamma} \mathbb C_{\vartheta}|_{\mathscr
        C^\prime})  = d_1 \quad \text{and} \quad
      \deg(\mathscr P\times_{\Gamma} \mathbb C_{\epsilon}|_{\mathscr
        C^\prime})  = d_2,
    \]
  \end{itemize}
  is bounded.
\end{theo}
See Definition~\ref{def:boundedness-G-bundles} and the paragraph below it for the precise meaning
of boundedness and explanation of some notation.

\medskip

\subsection{An example}
We illustrate our strategy by proving the theorem in the simplest nonabelian
case as a motivating example.

Let $\Gamma = GL_2 \times \mathbb C^*$ and $V = \mathbb C^{2\times 4}$ be the
space of $2 \times 4$ complex matrices. Let $GL_2$ act by left multiplication and
$\mathbb C^*$ act trivially. Let $\epsilon$ be the second projection, and
$\vartheta$ be the first projection followed by the determinant function. Thus
an LG-quasimap is simply a quasimap to the Grassmannian $G(2,4) =
V\git_{\theta} G$. Namely, given $\pi: \mathscr C \to S$ and $s\in
S(\mathbb C)$ as above, an LG-quasimap with domain $\mathscr C_s$ is equivalent
to a vector bundle $\mathscr E$ of rank $2$ on $\mathscr C_s$, together with a
section
\[
  \sigma = (\sigma_1 ,\ldots, \sigma_4) \in H^0(\mathscr C_s, \mathscr E^{\oplus
  4}),
\]
such that the set
\begin{equation}
  \label{eq:base-locus}
  \{q \in \mathscr C_{s} \mid \dim (\mathrm{span}(\sigma_1(q) ,\ldots, \sigma_4(q)))
  < 2\}
\end{equation}
is finite and away from the nodes and markings. For simplicity we fix $s$ and
assume $\mathscr C_s$ is a smooth projective curve.
The degree of such a quasimap is the first Chern class of $\mathscr E$.
We will show the boundedness of quasimaps from $\mathscr C_{s}$ to $G(2,4)$ of a fixed degree.

One might be tempted to use the well-known fact that the space of regular maps from
$\mathscr C_{s}$ into $G(2, 4)$ of a fixed degree is bounded.
Each quasimap does extend to regular map, since the base points are assumed to
be smooth points of $\mathscr C_{s}$.
However, two different quasimaps could induce the same regular map.
For example, take $\mathscr C_s = \mathbb P^1$ with homogeneous coordinates
$(x,y)$, $\mathscr E_{s} = \mathcal O(1)
\oplus \mathcal O(1)$, and take
\[
  \sigma_1(t) = (x - t y, x), \sigma_2(t) = (x, x+ ty), \sigma_3(t) =
  \sigma_4(t) = 0,\quad t \in \mathbb C^*.
\]
It is easy to see that for different $t$, they are different quasimaps.
However, the induced regular maps are the same constant map.

Thus, instead of bounding the induced regular maps, we will bound the underlying
vector bundles $\mathscr E_{s}$, and then the boundedness of quasimaps follow easily.

To bound such vector
bundles of a fixed degree $d$, we divide them into two cases. The first is when
$\mathscr E$ is slope semistable. It is well-known that the set of slope
semistable vector bundles of fixed rank and degree is bounded. Otherwise,
$\mathscr E$ is unstable. In this case, we
consider the Harder-Narasimhan filtration of $\mathscr E$
\[
  0 \to \mathscr L_1 \to \mathscr E \to \mathscr L_2 \to 0,
\]
where $\mathscr L_1$, $\mathscr L_2$ are line bundles with
\[
  \deg(\mathscr L_1)  > \deg(\mathscr L_2).
\]
Then \eqref{eq:base-locus} implies the image of $\sigma_1 ,\ldots, \sigma_4$ in $H^0(\mathscr C_s, \mathscr L_2)$
cannot be all trivial.
Hence
\begin{equation}
  \label{eq:L_2-nonnegative}
  \deg(\mathscr L_2) \geq 0.
\end{equation}
Since $\deg(\mathscr L_1) + \deg(\mathscr L_2) = d$ is fixed,
we see that both
$\deg(\mathscr L_1)$ and $\deg(\mathscr L_2)$ are bounded.
Finally, since the set of line bundles of a fixed degree is bounded and the set
of extensions of two given bounded families of line bundles is bounded, the set of all such vector
bundles $\mathscr E$ is bounded.

Note that in the above example, if we consider all filtrations on $\mathscr E$
instead of the Harder-Narasimhan filtration, the
resulting sets of $\mathscr L_1, \mathscr L_2$ are never bounded even for one
fixed $\mathscr E$, since
$\deg(\mathscr L_1)$ can be arbitrarily negative.

\medskip

For the general case, the notion of slope semistable vector bundles is replaced
by Ramanathan's semistable principal bundles, and the notion of Harder-Narasimhan filtration
is replaced by the canonical reduction of a principal bundle to a parabolic subgroup.
Generalizing \eqref{eq:L_2-nonnegative}, which uses the GIT stability
condition on the target space, will be the key step of the proof
(Lemmas~\ref{lem:positivity-of-quotient} and \ref{lem:convex-hull-contains-zero-in-interior}).

In Section~\ref{sec:preparation}, we review the notion of principal
bundles, their stability and canonical reductions. In
Section~\ref{sec:main-proof}, we prove the main theorem by showing that the
canonical reductions of the $\Gamma$-bundles have finitely many topological
types, except for the most technical lemma, which we prove in Section~\ref{sec:bounding-the-degrees}.

\medskip
\noindent
\textbf{Acknowledgments}
We would like to thank Professors Bumsig Kim, Yongbin Ruan, Jun Yu,
Rachel Webb, Dingxin Zhang for their valuable discussion with the authors.

Both
J.~Li and Y.~Zhou are partially supported by the National Key Res.\ and
Develop.\ Program of China \#2020YFA0713200.
H.-L.~Chang is partially supported by grants Hong Kong GRF16300720 and
GRF16303122. S.~Guo is partially supported by NSFC 12225101 and 11890660.  J.~Li is also partially
supported by NSFC 12071079 and
by Shanghai SF grants 22YS1400100. W.-P.~Li is partially supported by grants
Hong Kong GRF16304119, GRF16306222 and GRF16305125.
Y.~Zhou is also partially supported by Shanghai Sci.\ and Tech.\ Develop.\ Funds
\#22QA1400800 and
Shanghai Pilot Program for Basic Research-Fudan Univ.\ 21TQ1400100 (22TQ001).
Y.~Zhou would also like to thank the
support of Alibaba Group as a DAMO Academy Young Fellow, and would like to
thank the support of Xiaomi Corporation as a Xiaomi Young Fellow.

\iffalse
  \bibliographystyle{alpha}
  \bibliography{references.bib}
\fi

%%% Local Variables:
%%% mode: latex
%%% TeX-master: "boundedness-paper"
%%% End:

%% file: boundedness4.tex
\section{Preparation on principal bundles}

\label{sec:preparation}
Let $\mathbf G$ be an algebraic group, and $\pi: \mathscr C \to S$
be a family of smooth projective
curves over a scheme $S$ of finite type over $\mathbb C$.
For $s\in S(\mathbb C)$, we write $\mathrm{Bun}_{\mathbf G}(\mathscr C_{s})$ for
the set of isomorphism classes of principal $\mathbf G$-bundles on $\mathscr
C_{s} := \mathscr C|_{s}$. We call
\[
  \mathrm{Bun}_{\mathbf G}(\mathscr C/S) := \coprod_{s\in S(\mathbb C)}
  \mathrm{Bun}_{\mathbf G}(\mathscr C_{s})
\]
the set of principal $\mathbf G$-bundles on the fibers of $\pi$.
\begin{definition}
  \label{def:boundedness-G-bundles}
  A subset $\mathfrak  S \subset \mathrm{Bun}_{\mathbf G}(\mathscr C/S)$ is said
  to be bounded if there exists a finite type $S$-scheme $T$ and a principal
  $\mathbf G$-bundle $\mathscr G$ on $\mathscr C\times_{S} T$ such that for any
  $\mathscr G_{s} \in \mathfrak  S \cap \mathrm{Bun}_{\mathbf G}(\mathscr C_{s})$,
  $s\in S(\mathbb C)$, there exists $t\in T(\mathbb C)$ lying
  over $s$ such that $\mathscr G|_{t} \cong \mathscr G_{s}$.
\end{definition}

For a character $\chi \in \widehat{\mathbf G}$, denote by $\mathbb C_{\chi}$ the
$1$-dimensional $\mathbf G$-representation induced by $\chi$.
Thus for $\mathscr G_{s} \in \mathrm{Bun}_{\mathbf G}(\mathscr C_{s})$, the
Borel construction $\mathscr G_s \times_{\mathbf G} \mathbb C_{\chi}$ is a line
bundle on $\mathscr C_{s}$.
\begin{definition}
  \label{def:degree-principal-bundle}
  The degree  of $\mathscr G_{s}$ is defined to be the group homorphism
    $\mathrm{deg}_{\mathscr G_{s}} : \widehat{\mathbf G} \to \mathbb Z$ given by
  \[
    \mathrm{deg}_{\mathscr G_s}(\chi) = \deg_{\mathscr C_s}(\mathscr G_s
    \times_{\mathbf G} \mathbb C_{\chi}),\quad \forall \chi \in \widehat{\mathbf G}.
  \]
\end{definition}
For $d \in \mathrm{Hom}(\widehat{\mathbf G}, \mathbb Z)$, we denote by
\[
  \mathrm{Bun}_{\mathbf G}(\mathscr C/S, d) = \coprod_{s\in S(\mathbb C)}
  \mathrm{Bun}_{\mathbf G}(\mathscr C_{s}, d)
\]
the set of principal $\mathbf G$-bundles of degree $d$ on the fibers of $\pi$.

Given any morphism of algebraic groups
\[
  f : \mathbf G_1 \longrightarrow  \mathbf G_2,
\]
we define
\[
  \Psi_{f}: \mathrm{Bun}_{\mathbf G_1}(\mathscr C/S) \longrightarrow
  \mathrm{Bun}_{\mathbf G_2}(\mathscr C/S)
\]
to be the extension of structure groups along $f$, i.e.,
\[
  \Psi_{f} (\mathscr G_{s}) = \mathscr G_{s} \times_{\mathbf G_1} \mathbf G_2.
\]
We will also write $\Psi_{\mathbf G_1\to \mathbf G_2}$ if the map $f$ is clear from
the context.
Given $\mathscr G^\prime_{s} \in \mathrm{Bun}_{\mathbf G_2}(\mathscr C_s)$,
a reduction of its structure group to $\mathbf G_1$ is
a choice of $\mathscr G_{s} \in \mathrm{Bun}_{\mathbf G_1}(\mathscr C_s)$,
together with an isomorphism
\[
  \mathscr G_{s} \times_{\mathbf G_1} \mathbf G_2 \cong  \mathscr G^\prime_{s}
\]
as principal $\mathbf G_2$-bundles on $\mathscr C_s$.
When $f$ is a closed embedding, such a reduction is equivalent to a section of
$\mathscr G^\prime_{s}/ \mathbf G_{1} \to \mathscr C_{s}$.

From now on assume that $\mathbf G$ is a connected reductive group.
Fix
\[
  \mathbf T \subset \mathbf B \subset \mathbf G,
\]
where $\mathbf T$ is a maximal torus and $\mathbf B$ is a Borel subgroup.
Let $\mathbf P_1 = \mathbf B, \mathbf P_2,\ldots, \mathbf P_k$ be the parabolic
subgroups of $\mathbf G$ containing $\mathbf B$. Let $\mathbf L_i$ be the Levi
factor of $\mathbf P_i$.
\begin{definition}[\!{{\cite[Definition~1.1]{ramanathan1975stable}}}]
  \label{def:semistable-bundles}
  A $\mathscr G_{s} \in \mathrm{Bun}_{\mathbf G}(\mathscr C/S)$ is said to be
  semistable if for any $\mathbf P_i$ that is a maximal parabolic subgroup, and
  any reduction $\sigma: \mathscr C_s \to \mathscr G_{s}/\mathbf P_i$ of the
  structure group to $\mathbf P_i$, we have
  \[
    \deg(\sigma^*(T_{\mathscr G_s/\mathbf P_i})) \geq 0,
  \]
  where $T_{\mathscr G_s/\mathbf P_i}$ is the relative tangent bundle of
  $\mathscr G_s/\mathbf P_i \to \mathscr C_s$.
\end{definition}

For $d \in \mathrm{Hom}(\widehat{\mathbf G}, \mathbb Z)$,
we write
\[
  \mathrm{Bun}^{\mathrm{ss}}_{\mathbf G}(\mathscr C/S, d) = \coprod_{s\in S(\mathbb C)}
  \mathrm{Bun}^{\mathrm{ss}}_{\mathbf G}(\mathscr C_{s}, d)
\]
for the set of semistable principal $\mathbf G$-bundles of degree $d$ on the
fibers of $\pi : \mathscr C \to S$.

\begin{proposition}
  \label{prop:boundedness-semistable-G-bundles}
  $\mathrm{Bun}^{\mathrm{ss}}_{\mathbf G}(\mathscr C/S, d)$ is bounded.
\end{proposition}

The case when $S$ is a point is proved in \cite{ramanathan1996moduli} and
\cite{holla2001generalisation}. Their proofs generalize to the relative case
almost verbatim. For completeness, we include the necessary details below, following
\cite{holla2001generalisation}. Note that the Borel group in
\cite[Proposition~3.1]{holla2001generalisation} is generalized to any parabolic
group in Lemma~\ref{lem:from-L-bundles-to-G-bundles}. This will be useful for the
canonical reduction argument later.

Before proving this proposition, we first prove
\begin{lemm}
  \label{lem:from-L-bundles-to-G-bundles}
  For any $\mathfrak S\subset \mathrm{Bun}_{\mathbf P_i}(\mathscr C/S)$, if
  $\Psi_{\mathbf P_i\to \mathbf L_i}(\mathfrak S)$ is bounded, then $\Psi_{\mathbf
    P_i\to \mathbf G}(\mathfrak S)$ is bounded.
\end{lemm}
\begin{proof}
  The proof is similar to that of
  \cite[Proposition~3.1]{holla2001generalisation}.

  The first part ($\mathbf G=GL(n)$ case) is almost the same, except that the
  flag $\{E_i\}$ in \cite{holla2001generalisation} need not be full, since
  $\mathbf P_i$ is a parabolic subgroup.

  The second part (boundness of the associated $GL(n)$-bundles) needs slight
  modification. Write $\mathbf P = \mathbf P_i$ and $\mathbf L= \mathbf L_i$.
  Let $GL(n) = GL(W)$ where $W$ is some faithful representation of $\mathbf G$. Let
  $0\subset W_1 \subset W_2 \subset \cdots \subset W_n = W$ be a filtration by
  $\mathbf P$-invariant subspaces such that $U_{\mathbf P}$ acts trivially on the
  graded pieces $W_{i}/W_{i-1}$, where $U_{\mathbf P}$ is the unipotent radical
  of $\mathbf P$. Let $\mathbf P^\prime\subset GL(n)$ be the parabolic subgroup
  fixing the flag $\{W_i\}$. Let $U_{\mathbf P^\prime}$ be the unipotent radical of
  $\mathbf P^\prime$. Let $\mathbf L^\prime$ be the Levi factor of $\mathbf
  P^\prime$.
  Then by the construction of $\mathbf P^\prime$, $U_{\mathbf P}$ is mapped to
  $U_{\mathbf P^\prime}$, and we have the following commutative diagram similar
  to that in \cite{holla2001generalisation}.
  \[
    \begin{tikzcd}
      \mathbf G \arrow[r] & GL(n)\\
      \mathbf P \arrow[r] \arrow[d]\arrow[u, hook]& \mathbf P^\prime\arrow[u,hook] \arrow[d]\\
      \mathbf L \arrow[r] & \mathbf L^\prime.
    \end{tikzcd}
  \]
  The rest is similar. Start with a family of principal $\mathbf P$-bundles and
  extend the structure groups using the lower commutative square. If the
  associated $\mathbf L$-bundles are bounded, then the associated $\mathbf
  L^\prime$-bundles are bounded. Hence the associated $GL(n)$-bundles are bounded by
  the previous step.

  The third part (from the boundedness of $GL(n)$-bundles to the boundedness of
  $\mathbf G$-bundles) is verbatim.
\end{proof}

\begin{lemm}
  \label{lem:degree-bound-for-reduction-of-G-bundles}
  There exists a finite $\Omega \subset \mathrm{Hom}(\widehat {\mathbf T},
  \mathbb Z)$, such that each $\mathscr G_s \in
  \mathrm{Bun}^{\mathrm{ss}}_{\mathbf G}(\mathscr C/S, d)$ has a reduction to $\mathbf B$
  whose associated $\mathbf T$-bundle has degree in $\Omega$.
\end{lemm}
\begin{proof}
  This follows immediately from \cite{holla2001generalisation}.

  Indeed, the author showed that there exists a
  subgroup $L \subset \mathrm{Hom}(\widehat{\mathbf T}, \mathbb Z)$ of maximal
  rank, and for each $s\in S(\mathbb C)$ there exists a finite subset $\Omega_{s}
  \subset \mathrm{Hom}(L, \mathbb Z)$, such that each $\mathscr G_s
  \in \mathrm{Bun}_{\mathbf G}(\mathscr C_s)$ has a reduction
  to $\mathbf B$ with $\deg_{\mathscr G_{s,\mathbf T}}|_{L} \in \Omega_s$,
  where $\mathscr G_{s,\mathbf T}$ denotes the $\mathbf T$-bundle associated to
  the reduction of $\mathscr G_s$, via the extension of structure groups along
  $\mathbf B \to \mathbf T$. This
  is stated in the first paragraph of the proof of
  \cite[Theorem~1.2]{holla2001generalisation}, on page 331.
  Note that they used $\chi^*$ to denote the group of characters. Now, since $L$
  is of maximal rank, the restriction morphism $\mathrm{Hom}(\widehat{\mathbf
    T}, \mathbb Z) \to \mathrm{Hom}(L, \mathbb Z)$ is injective.
  We take
  \[
    \Omega = \bigcup_{s \in S(\mathbb C)} \Omega_{s} \cap
    \mathrm{Hom}(\widehat{\mathbf T}, \mathbb Z)
  \]
  and thus we have $\mathscr G_{s,\mathbf T}
  \in \Omega$.

  We verify that the $\Omega_{s}$ can be chosen to be independent of $s$, so
  that $\Omega$ is finite.
  Indeed, to bound the set of $\deg_{\mathscr G_{s, \mathbf T}}$ coming from
  appropriately chosen reduction of $\mathscr G_s$ as above, one
  bound comes from \cite[Theorem~1.1]{holla2001generalisation}, which states
  that there exists a reduction $\sigma$ such that $\deg(N_{\sigma}) \leq g \cdot
  \mathrm{dim}(\mathbf G/\mathbf B)$, where $N_\sigma$ is the
  $\sigma^*T_{\mathscr G_s/ \mathbf B}$ in
  Definition~\ref{def:semistable-bundles}.
  Another bound comes from
  \cite{ramanathan1996moduli}, which states that if $\mathbf G$ is semi-simple,
  then the degree of the line bundle associated to a dominant character of
  $\mathbf B$ is $\leq 0$. It is used in the last paragraph of the proof of
  \cite[Proposition~3.4]{holla2001generalisation}.
  All other bounds come from the definition of semistable principal bundles, which
  only involve the structure of $\mathbf G$.
\end{proof}

\begin{proof}[Proof for Proposition~\ref{prop:boundedness-semistable-G-bundles}]
  A principal $\mathbf T$-bundle is equivalent to a tuple of line bundles.
  The relative Picard functor for a fixed degree is bounded. Hence
  the lemma follows immediately from
  Lemma~\ref{lem:from-L-bundles-to-G-bundles}
  and
  Lemma~\ref{lem:degree-bound-for-reduction-of-G-bundles}.
\end{proof}

Recall that we have fixed the Borel $\mathbf B\subset \mathbf G$, and
$\mathbf P_1 = \mathbf B, \mathbf P_2,\ldots, \mathbf P_k$ are all the parabolic
subgroups of $\mathbf G$ containing $\mathbf B$.
\begin{defiprop}[{\!\cite[Proposition~1]{ramanathan1979moduli},\cite[pg.~202 and
    Theorem~4.1]{biswas2004harder}}]
  \label{def:canonical-reductions}
  A reduction $\mathscr Q_s\in \mathrm{Bun}_{\mathbf P_i}(\mathscr C_s)$ of
  $\mathscr G_{s} \in \mathrm{Bun}_{\mathbf G}(\mathscr C_s)$ is called a
  canonical reduction, or a Harder-Narasimhan reduction, if the following two
  conditions hold:
  \begin{enumerate}
  \item the principal $\mathbf L_i$-bundle $\mathscr Q_s \times_{\mathbf P_i}
    \mathbf L_i$ is semistable;
  \item
    for any nontrivial character $\chi$ of $\mathbf P_i$ which is a
    nonnegative linear combination of simple roots, $\deg_{\mathscr Q_{s}}(\chi)
    >0$.
  \end{enumerate}
  Up to isomorphisms, any $\mathscr G_s\in \mathrm{Bun}_{\mathbf G}(\mathscr
  C_s)$ admits a unique canonical reduction to a unique $\mathbf P_i$.
\end{defiprop}
Now, for $\mathscr G_{s} \in \mathrm{Bun}_{\mathbf G}(\mathscr C_s)$, denote its canonical reduction by
\[
  \mathscr G^{\mathrm{can}}_s \in \coprod_{i=1}^k \mathrm{Bun}_{\mathbf
    P_i}(\mathscr C_{s}).
\]
\begin{lemm}
  \label{lem:bounding-by-degree-of-L}
  Let $\mathfrak S\subset \mathrm{Bun}_{\mathbf G}(\mathscr C/S)$.
  Suppose that
  \[
    \{\deg_{\mathscr G^{\mathrm{can}}_s} \mid \mathscr G_s \in \mathfrak S\}
    \subset \coprod_{i=1}^k \mathrm{Hom}(\widehat{\mathbf P}_i, \mathbb Z)
  \]
  is finite.
  Then $\mathfrak S$ is bounded.
\end{lemm}
\begin{proof}
  Observe that $\widehat{\mathbf L}_i = \widehat{\mathbf P}_i$ and
  $\deg_{\mathscr G^{\mathrm{can}}_s} = \deg_{\mathscr G^{\mathrm{can}}_s\times_{\mathbf P_i} \mathbf L_i}$.
  The set of associated
  $\mathbf L_i$-bundles, which are semistable by the definition of canonical
  reduction, is bounded by Proposition~\ref{prop:boundedness-semistable-G-bundles}.
  Then the lemma follows from  Lemma~\ref{lem:from-L-bundles-to-G-bundles}.
\end{proof}

Now we consider  disconnected reductive groups.
Let $\mathbf G^\prime$ be an algebraic group containing $\mathbf G$ as its
identity component.
For $\mathscr G_s^\prime \in \mathrm{Bun}_{\mathbf G^\prime}(\mathscr C_{s})$,
$\mathscr G_s^\prime/\mathbf G \to \mathscr C_{s}$ is a $\mathbf G^\prime/ \mathbf G$-bundle and
$\mathscr G_s^\prime \to \mathscr G_s^\prime/\mathbf G$ is a $\mathbf G$-bundle.
Thus $\mathscr G_s^\prime/\mathbf G$ is a smooth projective not necessarily connected curve.
\begin{lemm}
  \label{lem:boundedness-finite-cover}
  Given any finite subset
  \[
    \Omega \subset \coprod_{i=1}^k \mathrm{Hom}(\widehat{\mathbf P}_i, \mathbb Z),
  \]
  the set
  \begin{equation}
    \label{eq:family-of-finite-covers-and-G-bundles}
    \begin{aligned}
      \{ \mathscr G^\prime_s \to \mathscr C^\prime_s \to  \mathscr C_s \mid
      & \mathscr C^\prime_s \to \mathscr C_s \text{ is a principal }\mathbf
        G^\prime/\mathbf G\text{-bundle}, \\
      & \mathscr G^\prime_s \to \mathscr C^\prime_s \text{ is a principal }\mathbf
        G\text{-bundle}, s \in S(\mathbb C), \text{such that }\\
      & \text{for each connected component }\tilde{\mathscr C}_{s} \subset \mathscr C_{s}^\prime,
        \text{ we have}\\
      & \deg_{
        {(
        \mathscr G^\prime_s|_{\tilde{\mathscr C}_{s}}
        )}^{\mathrm{can}}
        }
        \in \Omega
        \}
    \end{aligned}
  \end{equation}
  is bounded.
\end{lemm}
\begin{proof}
  The genus of $\mathscr C^\prime_s$ is bounded. Given two bounded family of
  curves, the set of isomorphisms between their fibers is bounded. Hence the set of
  principal $\mathbf G^\prime/\mathbf G$-bundles $\mathscr C^\prime_s \to
  \mathscr C_s$ is bounded. Hence the lemma follows from
  Lemma~\ref{lem:bounding-by-degree-of-L}.
\end{proof}

\begin{lemm}
  \label{lem:bounding-by-degree-of-L-disconnected-case}
  Let $\mathfrak  S \subset \mathrm{Bun}_{\mathbf G^\prime}(\mathscr C/ S)$.
  Suppose that
  \begin{equation}
    \label{eq:degree-of-bundles-over-finite-cover}
    \begin{aligned}
      \{
      \deg_{
      {(
      \mathscr G^\prime_s|_{\tilde{\mathscr C}_{s}}
      )}^{\mathrm{can}}
      }
      \mid & \mathscr G^\prime_{s} \in \mathfrak S\cap
             \mathrm{Bun}_{\mathbf G^\prime}(\mathscr C_{s}), s \in S(\mathbb C),\\
           & \tilde{\mathscr C}_{s}\text{ is a
             connected component of }\mathscr G^\prime_{s}/\mathbf G
             \}
    \end{aligned}
  \end{equation}
  is finite, then $\mathfrak S$ is bounded.
\end{lemm}
\begin{proof}
  Apply Lemma~\ref{lem:boundedness-finite-cover} with $\Omega$ equals to
  \eqref{eq:degree-of-bundles-over-finite-cover}. After replacing $S$ by another
  finite-type $S$-scheme, we may assume that there exists
  \begin{equation}
    \label{eq:G-prime-to-C-prime-to-C-to-S}
    \mathscr G^\prime \longrightarrow
    \mathscr C^\prime \longrightarrow  \mathscr C \longrightarrow S
  \end{equation}
  whose fibers over all $s\in S(\mathbb C)$ contain the set
  \eqref{eq:family-of-finite-covers-and-G-bundles}.
  Choose a closed point $g_i$ of each connected component of $\mathbf G^\prime$,
  $i  =1 ,\ldots, n$.
  For any $s\in S(\mathbb C)$, we have a bijection between
  \begin{enumerate}
  \item
    a $\mathbf G^\prime$-bundle structure on $\mathscr
    G_s^\prime \to \mathscr C_s$ such that the induced
    \[
      \mathscr G_s^\prime  \longrightarrow \mathscr G_s^\prime/ \mathbf G
      \longrightarrow  \mathscr C_s \longrightarrow \{s\}
    \]
    is isomorphic to the restriction of \eqref{eq:G-prime-to-C-prime-to-C-to-S} to
    $s$, compatible with the $\mathbf G$ action on $\mathscr G_s^{\prime}$ and
    the $\mathbf G^\prime/\mathbf G$ action on $\mathscr C_s^\prime \cong \mathscr
    G_s^\prime/ \mathbf G$;
  \item
    an isomorphism $f_i: \mathscr G_{s}^\prime \to \mathscr G_{s}^\prime$ for each
    $i = 1 ,\ldots, n$,
    making the following diagram commute for each $g_0\in \mathbf G(\mathbb C)$
    \[
      \begin{tikzcd}
        \mathscr G^\prime_{s} \ar[r, "g_0"] \ar[d, "f_i"] & \mathscr
        G^\prime_{s}\ar[d,"f_i"] \ar[r] & \mathscr C_s^\prime \ar[d, "g_i"]\\
        \mathscr G^\prime_{s} \ar[r, "g_ig_0g_i^{-1}"] & \mathscr
        G^\prime_{s}\ar[r] & \mathscr C_s^\prime
      \end{tikzcd}.
    \]
  \end{enumerate}
  The bijection is given by letting $g_i$ act by $f_i$.

  It suffices to show that the set of such $f_i$ is bounded.
  Such an $f_i$ is an isomorphism of principal $\mathbf G$-bundles over $\mathscr
  C_s^{\prime}$
  \[
    \mathscr G^\prime_{s} \longrightarrow g_i^*(\mathscr G^\prime_{s})^{g_i}
  \]
  where $(\mathscr G^\prime_{s})^{g_i}$ is $\mathscr G^\prime_{s}$ with the
  original $\mathbf G$-action conjugated by $g_i$. Since the set of isomorphisms between two
  bounded famillies of principal $\mathbf G$-bundles is bounded, the proof is complete.
\end{proof}

\section{The boundedness of stable LG-quasimaps}
\label{sec:main-proof}
We have finished the digression on the generality of principal bundles.
We now apply the general theory in the previous section to prove
Theorem~\ref{thm:main}. So we turn back to the notation in Section~\ref{sec:intro}

Let $\Gamma_0 \subset \Gamma$ be the identity component and set $G_0 = G\cap
\Gamma_0$. Then $V^{\mathrm{s}}(\theta) = V^{\mathrm{s}}(\theta|_{G_0})$.
Fix a Borel subgroup $B\subset \Gamma_0$ and let $P_1 = B, \ldots, P_k \subset
\Gamma_0$ be the parabolic subgroups containing $B$.
Note that in applying the general theory, the $\mathbf G$ in
Lemma~\ref{lem:bounding-by-degree-of-L}
will be $\Gamma_0$, while the $\mathbf G$ in
Proposition~\ref{prop:boundedness-semistable-G-bundles} will be the Levi
factor of some $P_i$.

\smallskip
The following lemma is the key, and will be proved in Section~\ref{sec:bounding-the-degrees}.
\begin{lemm}
  \label{lem:boundedness-after-reduction-2}
  For each $i = 1 ,\ldots, k$ and for any integers $d_1,d_2,g$ with $ g \geq 2$,
  there exists a finite subset $\Omega \subset \mathrm{Hom}(\widehat{P}_i, \mathbb
  Z)$, such that for any smooth projective curve $\mathscr C$ of genus $g$ and
any principal $P_i$-bundle
  $\mathscr Q$ on $\mathscr C$ satisfying
  \begin{itemize}
  \item $\mathscr Q$ is the canonical reduction of the principal
    $\Gamma_0$-bundle $\mathscr Q\times_{P_i} \Gamma_0$,
  \item $\deg_{\mathscr Q}(\vartheta|_{P_i}) = d_1$, $\deg_{\mathscr
      Q}(\epsilon|_{P_i}) = d_2$, and
  \item $\mathscr Q \times_{P_i} V \to \mathscr C$ admits a section $\sigma$
    such that $\sigma^{-1}(\mathscr Q\times_{P_i} V^{\mathrm{s}}(\theta)) \neq \emptyset$,
  \end{itemize}
  we have $\deg_{\mathscr Q} \in \Omega$.
\end{lemm}
Assuming this, we continue to prove the boundedness for stable LG-quasimaps.
Combinging Lemma~\ref{lem:bounding-by-degree-of-L-disconnected-case} and the
above lemma, it is easy to obtain
\begin{lemm}
  \label{lem:boundedness-bundles-non-orbifold}
  Let $\pi: \mathscr C \to S$ be a family of smooth projective curves of genus
  $g\geq 2$ over a finite-type $\mathbb C$-scheme $S$.
  Then for any $d_1, d_2 \in \mathbb Z$, the set
  \begin{equation}
    \label{eq:underlying-Gamma-bundle-on-smooth-curves}
    \Bigg\{\mathscr P_{s} \in \mathrm{Bun}_{\Gamma} (\mathscr C_{s}) \Bigg|
    \begin{aligned}
      & s \in S(\mathbb C), \deg_{\mathscr P_s}(\vartheta) = d_1, \deg_{\mathscr
        P_s}(\epsilon) = d_2,
        \mathscr P_{s}\times_{\Gamma} V \text{ } \\
&
  \text{admits a section }
        \sigma \text{
        s.t. }\sigma^{-1}(\mathscr P_{s}\times_{\Gamma} V^{\mathrm{s}}(\theta)) \neq \emptyset
    \end{aligned}
    \Bigg\}
  \end{equation}
  is bounded.
\end{lemm}
\begin{proof}

  Let $\mathscr P_{s}$ be in the above set, and let $\tilde{\mathscr C}^{\prime}_s \subset
  \mathscr C^{\prime}_s := \mathscr P_{s} / \Gamma_0$ be a connected component.
  Form the principal $\Gamma_0$-bundle
  $\tilde{\mathscr P}^\prime_{s} = \mathscr P_{s} \times_{\mathscr C_{s}}
  \tilde{\mathscr C}^{\prime}_s$  over $\tilde{\mathscr C}^{\prime}_s$ and let
  $\mathscr Q_s$ be its canonical reduction to some (unique) $P_i$. Then we have
  $\deg_{\mathscr Q_s} (\vartheta|_{P_i}) = [\tilde{\mathscr C}^{\prime}_s:
  {\mathscr C}_s]\cdot d_1$ and $\deg_{\mathscr Q_s} (\epsilon|_{P_i}) =
  [\tilde{\mathscr C}^{\prime}_s: {\mathscr C}_s]\cdot d_2$, where
  $[\tilde{\mathscr C}^{\prime}_s: {\mathscr C}_s]$ is the degree of the finite
  \'etale cover $\tilde{\mathscr C}^{\prime}_s \to {\mathscr C}_s$. The section
  $\sigma$ pulls back to a section of $\mathscr Q_s \times_{P_i} V$ that takes the
  generic of $\tilde{\mathscr C}^{\prime}_s$ to $\mathscr Q_s \times_{P_i}
  V^{\mathrm{s}}(\theta)$. Hence
  by Lemma~\ref{lem:boundedness-after-reduction-2}, all the $\deg_{\mathscr
    Q_s}$ form a finite subset of $\mathrm{Hom}(\widehat P_i,\mathbb Z)$. Hence
  by Lemma~\ref{lem:bounding-by-degree-of-L-disconnected-case}, the set
  \eqref{eq:underlying-Gamma-bundle-on-smooth-curves} is
  bounded.
\end{proof}

To generalize the above lemma to twisted curves of any genus, we need the
following lemma.
\begin{lemm}
  \label{lem:covering}
  Let $\pi: \mathscr C \to S$ be a family of smooth twisted curves over a
  finite-type $\mathbb C$-scheme $S$.
  Then up to base change to an \'etale cover of $S$, there exists a family of smooth
  projective curves $\mathscr C^\prime \to S$ with connected fibers of genus
  $\geq 2$, together with a faithfully flat $S$-morphism $\mathscr C^\prime
  \to \mathscr C$.
\end{lemm}
\begin{proof}
  This is basically Lemma~6.7 of \cite{mixed-fields-main}, note that by
  Riemann-Hurwitz the genus of $\mathscr C^\prime$ can be made arbitrarily large.
\end{proof}

\begin{proof}[Proof of Theorem~\ref{thm:main}]
  The set of
  isomorphisms between two given bounded families of
  principal bundles on projective curves is bounded.
  Hence the lemma follows immediately from
  Lemma~\ref{lem:boundedness-bundles-non-orbifold} by a standard descent argument using
  Lemma~\ref{lem:covering}.
\end{proof}

\section{Proof for Lemma~\ref{lem:boundedness-after-reduction-2}}
\label{sec:bounding-the-degrees}
From now on we work with a single smooth projective curve $\mathscr C$ of fixed genus $g \geq 2$.
Recall that $\Gamma_0 \subset \Gamma$ is the identity component and $G_0 = G\cap
\Gamma_0 = \ker(\epsilon|_{\Gamma_0})$, and
the stable locus $V^{\mathrm{s}}(\theta)$ is nonempty.
If $V \hookrightarrow W$ is an equivariant embedding into a $\Gamma$-representation, then
we have $V^{\mathrm{s}}(\theta) = W^{\mathrm{s}}(\theta) \cap V$. Replacing
$V$ by $W$, we may assume that $V$ itself is a $\Gamma$-representation.
Since the lemma only involves $\Gamma_0$ and $G_0$, without loss of generality
we may assume $\Gamma = \Gamma_0$ and $G = G_0$.
We fix one $i = 1 ,\ldots, k$ and set $P = P_i$.

The strategy is to define $\Omega$ via the following lemma.
For
\[
  K \subset \widehat P_{\mathbb R}:= \widehat{P} \otimes_{\mathbb Z} \mathbb R
  \quad \text{and} \quad \mathcal B \in \mathbb R,
\]
we define
\[
  \Omega(K, \mathcal B) := \{d \in \mathrm{Hom}(\widehat P, \mathbb Z) \mid d(\chi) \geq
  \mathcal B, ~\forall \chi \in K\}.
\]
\begin{lemm}
  \label{lem:finiteness-Omega-K-B}
  Suppose the interior of the convex hull of $K$ contains
  $0$, then $\Omega(K, \mathcal B)$ is finite.
\end{lemm}
\begin{proof}
  This is elementary and we omit the proof.
\end{proof}

We will find finitely many $K_i$ such that for
sufficiently negative $\mathcal B$,
\begin{equation}
  \label{eq:Omega}
  \Omega := \bigcup \Omega(K_i, \mathcal B)
\end{equation}
contains each $\deg_{\mathscr Q}$ in
Lemma~\ref{lem:boundedness-after-reduction-2}. By the above lemma, the proof of
Lemma~\ref{lem:boundedness-after-reduction-2} will be complete once we show that
the interior of the convex hull of each $K_i$ contains $0$.

\smallskip
To begin with, we first look at those $\chi \in \widehat{\Gamma}$. This step is
similar to the proof in the abelian case in \cite{mixed-fields-main}.

We first choose $\theta_i \in \widehat{G}$, $i=1 ,\ldots, n$, such that
\begin{itemize}
\item $V^{\mathrm{s}}(\theta) \subset V^{\mathrm{ss}}(\theta_i)$, where
\item $\theta \in \sum_{i} \mathbb R_{> 0} \theta_i$,
\item $\theta_1 ,\ldots, \theta_n$ span $\widehat{G} \otimes \mathbb R$.
\end{itemize}
This is possible because any $\theta_i$ such that the ray $\mathbb R_{\geq
  0}\theta_i$ is sufficiently close to $\mathbb R_{\geq 0}\theta$ satisfies the
first requirement.

Note that by the exact sequence \eqref{eq:extension-of-groups},
$\Gamma$ is also reductive and the restriction map $\widehat{\Gamma} \to
\widehat{G}$ is surjective.
We pick a lift $\vartheta_i\in \widehat{\Gamma}$ of
$\theta_i$, for each $i = 1 ,\ldots, n$.
\begin{lemm}
  \label{lem:bounding-theta_i}
  For $\chi = -\vartheta, \epsilon, -\epsilon, \vartheta_1 ,\ldots,
  \vartheta_n$,
  \[
    \deg_{\mathscr Q}(\chi) := \deg (\mathscr Q\times_{P} \mathbb C_{\chi|_{P}})
  \]
  is uniformly bounded from below for those $\mathscr Q$ in
  Lemma~\ref{lem:boundedness-after-reduction-2}.
\end{lemm}
\begin{proof}
  This is identical to Lemma~6.8 of \cite{mixed-fields-main}.
  Note that for $\chi = -\vartheta, \epsilon$ or $-\epsilon$, $\deg_{\mathscr
    Q}(\chi)$ is constant.
\end{proof}

\begin{lemm}
  \label{lem:convex-hull-char-of-gamma}
  The convex hull of $\{-\vartheta, \epsilon, -\epsilon, \vartheta_1 ,\ldots,
  \vartheta_n\}$ in $\widehat{\Gamma}_{\mathbb R}$ contains the origin in its interior.
\end{lemm}
\begin{proof}
  The kernel of the surjection $\widehat{\Gamma}_{\mathbb R} \to
  \widehat{G}_{\mathbb R}$ is the span of $\epsilon$. The lemma follows
  immediately from the choices of those characters.
\end{proof}

We then look at those $\chi \in \widehat{P}$ that do not come from
$\widehat{\Gamma}$. Recall that $B\subset P$ is a Borel subgroup of $\Gamma =
\Gamma_0$. We fix a maximal torus $T \subset B$. Let $\mathfrak g, \mathfrak p,
\mathfrak b, \mathfrak t, \mathfrak l$ be the Lie algebras of $\Gamma, P, B, T,
L$, respectively. Let $\Phi \subset \widehat T$ be the roots of $\Gamma$, and
let $\Phi^+$, $\Phi^{-}$, $\Delta$ be the set of positive, negative, and simple
roots, respectively. Let $\Phi_{P}$ (resp.\ $\Phi_{L}$) be the set of nonzero weights
in $\mathfrak p$ (resp.\ $\mathfrak l$). Note that the composition of $T \to P
\to L$ is an embedding, via which $T$ is also a maximal torus of $L$. Thus
$\Phi_{L}$ is the roots of $L$, and we have
\[
  \Phi_{L} = \Phi_P \cap (- \Phi_P).
\]
We choose $\Phi_{L}^+ := \Phi_{L} \cap \Phi^+$ to be the positive roots of
$L$.
Then
\[
  \Delta_{L} := \Delta \cap \Phi_{L}
\]
are the simple roots of $L$.

Let $\mathbb R\Phi$ be the span of $\Phi$ in $\widehat{T}_{\mathbb R}$. We denote the Killing form on it by
$(\cdot ,\cdot )$. Let $\Delta^{\perp}_{L}$ be the orthogonal complement of
$\Delta_{L}$ in $\mathbb R\Phi$.
Then we have
\begin{equation}
  \label{eq:char-L}
  \widehat{T}_{\mathbb R} = \widehat{\Gamma} \oplus \mathbb R \Phi \quad \text{and} \quad
  \widehat{L}_{\mathbb R} \cong \widehat{\Gamma}_{\mathbb R} \oplus \Delta^{\perp}_{L}.
\end{equation}

\begin{lemm}
  \label{lem:positivity-projection-of-simple-root}
  Let $\pi: \mathbb R\Phi \to \Delta^\perp_L$ be the orthogonal projection, and
  $\alpha \in \Delta$ be a simple root.
  Then for any $\mathscr Q$ in Lemma~\ref{lem:boundedness-after-reduction-2},
  we have
  \[
    \deg_{\mathscr Q}(\pi(\alpha)) \geq 0.
  \]
\end{lemm}
\begin{proof}
  Recall that for two distinct simple roots $\alpha_i, \alpha_j$,
  \[
    \frac{(\alpha_i, \alpha_j)}{(\alpha_i, \alpha_i)} \in \mathbb Q_{\leq 0}.
  \]
  Analyzing the Gram-Schmidt process, we see that $\pi(\alpha)$ is a linear
  combination of simple roots with $\mathbb Q_{\geq 0}$-coefficients.
  Hence some multiple of $\pi(\alpha)$ is a character of $L$ that is a
  nonnegative linear combination of simple roots. Recall that $\widehat{L} =
  \widehat{P}$.
  Then $\deg_{\mathscr Q}(\pi(\alpha)) \geq 0$ since $\mathscr Q$ is the
  canonical reduction of $\mathscr Q\times_{P} \Gamma$
  (Definition~\ref{def:canonical-reductions}).
\end{proof}
\begin{lemm}
  \label{lem:lower-bound-det-with-section}
  Let $W$ be an irreducible $L$-representation and $\mathscr Q^\prime$ be a
  semistable principal $L$-bundle on $\mathscr C$, if $\mathscr Q^\prime
  \times_{L} W$ admits a nontrivial section, then
  \[
    \textstyle
    \deg(\mathscr Q^\prime \times_{L} \bigwedge^{\mathrm{top}} W)) \geq 0.
  \]
\end{lemm}
\begin{proof}
  By Schur's lemma, the representation $L \to GL(W)$ takes the center of $L$
  into the center of $GL(W)$. Hence by
  \cite[Theorem~3.18]{remanan1984remarks}, the principal $GL(W)$-bundle $\mathscr Q^\prime
  \times_{L} GL(W)$ is semistable. Hence the vector bundle $\mathscr Q^\prime
  \times_{L} W$ is slope semistable. Thus having a nontrivial section implies
  that its degree is nonnegative.
\end{proof}

Let
\[
  V = \bigoplus_{\alpha \in \widehat{T}} V_{\alpha}
\]
be the decomposition of $V$
into weight spaces.
For $x \in V$, we write $x_{\alpha}$ for its component in $V_{\alpha}$.
\begin{lemm}
  \label{lem:positivity-of-quotient}
  Given any $(\mathscr Q, \sigma)$ in Lemma~\ref{lem:boundedness-after-reduction-2},
  suppose that $x\in V$ whose $P$-orbit is in the image of $\sigma$, and
  $x_\alpha \neq 0$ for some $\alpha\in \widehat{T}$.
  Let
  \[
    \Lambda = \mathbb Z_{\geq 0} \Phi^{-} + \mathbb Z \Delta_{L} \subset \widehat{T}.
  \]
  Then there exists some $\alpha^\prime \in (\mathbb Z_{\geq
    1} \alpha + \Lambda) \cap \widehat{P}$ such that
  \[
    \deg_{\mathscr Q}(\alpha^\prime)
    \geq 0.
  \]
  Moreover, as $\mathscr C, \mathscr Q, \sigma$ and $x$ vary, the
  $\alpha^\prime$ can be taken from a fixed finite set.
\end{lemm}
\begin{proof}
  First note that $\Phi_P$ is generated by $\Phi^+$ and $\Delta_L$.
  Let
  \[
    V_+ = \bigoplus_{\beta\not\in \alpha+\Lambda} V_{\beta}.
  \]
  Note that the set $\{\beta\mid \beta\not \in \alpha + \Lambda\}$ is invariant
  under translation by $\Phi_{P}$. Hence $V_+$ is $P$-invariant. Consider the
  $P$-representation $V_P = V/V_+$. Since $x_\alpha\neq 0$, the image of $x$ in
  $V_P$ is nonzero.
  Hence under the quotient map
  \[
    \mathscr Q\times_{P} V \longrightarrow \mathscr Q\times_{P} V_P
  \]
  $\sigma$ is mapped to a nontrivial section of
  the vector bundle $\mathscr Q\times_{P} V_P$.

  Take a filtration $F_iV_{P}$ on $V_P$ by $P$-invariant subspaces such that the
  unipotent radical of $P$ acts trivially on the graded pieces. Thus the
  graded pieces are $L$-representations. This
  induces a filtration on $\mathscr Q \times_{P} V_P$. Since $\mathscr Q
  \times_{P} V_P$ has a nontrivial section, at least one of
  the graded pieces admits a nontrivial section.

  Say $\mathscr Q\times_{P}\mathrm{gr}_i V_P$ admits a nontrivial section. We
  further decompose $\mathrm{gr}_i V_P = \bigoplus_{j} V_{L,j}$ as a direct sum of
  irreducible $L$-representations. Then for at least one $j$, the vector bundle
  $\mathscr Q\times_{P} V_{L,j} = \mathscr Q\times_{P} L \times_{L} V_{L,j}$
  admits a nontrivial section.
  Since $\mathscr Q$ is a canonical reduction of $\mathscr Q\times_{P}\Gamma$,
  the principal $L$-bundle $\mathscr Q\times_{P} L$ is semistable.
  Applying Lemma~\ref{lem:lower-bound-det-with-section},
  we conclude that
  \[
    \textstyle
    \deg (\mathscr Q\times_{P} \bigwedge^{\mathrm{top}} V_{L,j})
    \geq 0.
  \]

  Let $\alpha^\prime \in \widehat{P}$ be such that $\bigwedge^{\mathrm{top}}
  V_{L,j} \cong \mathbb C_{\alpha^\prime}$ as $P$-representations. To show that
  this is the desired $\alpha^\prime$, it remains to show that $\alpha^\prime \in
  \mathbb Z_{\geq 1} \alpha + \Lambda$. Indeed, by construction, the
  weights of $V_{L,j}$ are contained in $\alpha+\Lambda$, since it is a
  subquotient of $V_P$. Since $\alpha^\prime$ is the sum of the weights
  in $V_{L,j}$, we have $\alpha^\prime \in \mathbb Z_{\geq 1} \alpha +
  \Lambda$ as desired. From this construction it is clear that $\alpha^\prime$
  is taken from a fixed finite set.
\end{proof}

For any $\mathscr Q, \sigma$ in Lemma~\ref{lem:boundedness-after-reduction-2},
let $K_{\mathscr Q, \sigma} \subset \widehat{P}_{\mathbb R}$ be the set consisting of
\begin{enumerate}
\item
  all the characters in Lemma~\ref{lem:convex-hull-char-of-gamma};
\item
  the orthogonal projection of each simple root onto $\Delta_{L}^\perp$;
\item
  all the $\alpha^\prime$ (for varying $x$ and $\alpha$) from
  Lemma~\ref{lem:positivity-of-quotient}, taken from a fixed finite set.
\end{enumerate}
As
$\mathscr C, \mathscr Q, \sigma$ vary, there are only finitely many
such distinct $K_{\mathscr Q, \sigma} \subset \widehat{P}_{\mathbb R}$.
To prove Lemma~\ref{lem:boundedness-after-reduction-2}, we define the desired $\Omega$
to be the finite union
\[
  \Omega := \bigcup_{\mathscr Q, \sigma}
  \{d\in \mathrm{Hom}(\widehat{P}, \mathbb Z) \mid d(\chi) \geq \mathcal B,\
  \forall \chi \in K_{\mathscr Q, \sigma}\},
\]
where
$\mathcal B$ is taken to be the
minimum of $0$ and the lower bound in Lemma~\ref{lem:bounding-theta_i}.
Then by
Lemma~\ref{lem:bounding-theta_i},
Lemma~\ref{lem:positivity-projection-of-simple-root}, and
Lemma~\ref{lem:positivity-of-quotient}, it
is clear that for every $\mathscr Q$ in
Lemma~\ref{lem:boundedness-after-reduction-2}, $\deg_{\mathscr Q} \in \Omega$.
It remains to show that each
\[
  \Omega(K_{\mathscr Q, \sigma}, \mathcal B) :=
  \{d\in \mathrm{Hom}(\widehat{P}, \mathbb Z) \mid d(\chi) \geq \mathcal B,\
  \forall \chi \in K_{\mathscr Q, \sigma}\}
\]
is finite.
By Lemma~\ref{lem:finiteness-Omega-K-B}, to prove
Lemma~\ref{lem:boundedness-after-reduction-2} it suffices to show that
\begin{lemm}
  \label{lem:convex-hull-contains-zero-in-interior}
  For any given $\mathscr Q, \sigma$ in
  Lemma~\ref{lem:boundedness-after-reduction-2}, the convex hull
  of $K_{\mathscr Q,
    \sigma}$ in $\widehat{P}_{\mathbb R} = \widehat{L}_{\mathbb R}$ contains $0$
  in its interior.
\end{lemm}
\begin{proof}
  We will prove it by showing that $K_{\mathscr Q,
    \sigma}$ cannot be contained in any half space.
  Namely, suppose we have a nonzero a linear functional
  \[
    \varphi: \widehat{P}_{\mathbb R} = \widehat{L}_{\mathbb R} \longrightarrow \mathbb R \quad
    \text{such that  }
    \varphi(K( \mathscr Q, \sigma)) \subset \mathbb R_{\geq 0},
  \]
  we would like to get a contradiction.

  Given such $\varphi$, using \eqref{eq:char-L}, we extend
  it to a functional on $\widehat{T}_{\mathbb R}$, which we still call $\varphi$, by setting
  $\varphi|_{\Delta_{L}} = 0$. By Lemma~\ref{lem:convex-hull-char-of-gamma},
  $\varphi|_{\widehat{\Gamma}_{\mathbb R}} = 0$. For any simple root
  $\alpha \in \Delta$,
  we have $\varphi(\alpha)  = \varphi(\pi(\alpha))\geq 0$, since
  $\pi(\alpha)$ is in $K( \mathscr Q, \sigma)$.
  Hence
  $\varphi(\Lambda)
  \subset \mathbb R_{\leq 0} $, where $\Lambda$ is
  defined in Lemma~\ref{lem:positivity-of-quotient}.
  Thus, for any $\alpha$ in Lemma~\ref{lem:positivity-of-quotient}, we must
  have $\varphi(\alpha) \geq 0$. Since there are only finitely many $\alpha$,
  if $\varphi \neq 0$, we can replace $\varphi$ by a linear functional defined
  over $\mathbb Q$ with the same properties.
  Thus there is a $1$-parameter subgroups
  \[
    \lambda: \mathbb G_m \longrightarrow  T,
  \]
  such that
  $\langle \lambda ,\alpha \rangle \geq 0$ whenever $x_{\alpha} \neq 0$, and
  $\langle \lambda ,\beta \rangle = 0$ for any $\beta \in \widehat\Gamma$.
  In particular, it violates the Hilbert-Mumford
  criterion \cite[Proposition~2.5]{king1994moduli} and thus no such $x$ can be
  stable.
  This contradicts to the last assumption in
  Lemma~\ref{lem:boundedness-after-reduction-2}.
\end{proof}

\iffalse
\bibliography{references.bib}
\fi
%%% Local Variables:
%%% mode: latex
%%% TeX-master: "boundedness-paper"
%%% End:

% LocalWords:  semistable